\documentclass[11pt]{article}

\usepackage{amsmath,amssymb,amsfonts}
\usepackage{xcolor}
\usepackage{hyperref}
\usepackage{nicefrac}
\usepackage{enumitem}
\usepackage{fullpage}

\newcommand{\E}{{\rm \bf E}}
\newcommand{\prob}{{\rm \bf P}}
\newcommand{\calF}{{\cal F}}

\newtheorem{theorem}{Theorem}

\newtheorem{lemma}[theorem]{Lemma}

\newtheorem{remark}[theorem]{Remark}

\newenvironment{proof}[1][Proof]{\noindent\textbf{#1.} }{\ \rule{0.5em}{0.5em}}

\date{}

\title{About a ball removal process on bins}

\author{Correa, José%
\footnote{Departamento de Ingeniería Industrial and Ctr.~de Modelamiento Matemático (CNRS IRL2807), Universidad de Chile, Santiago, Chile. \tt{correa@uchile.cl}} \and 
Kiwi, Marcos%
\footnote{Departamento de Ingeniería Matemática and Ctr.~de Modelamiento Matemático (CNRS IRL2807),
Universidad de Chile, Santiago, Chile.
\tt{mk@dim.uchile.cl}} \and 
Livanos, Vasilis%
\footnote{Ctr.~de
Modelamiento Matemático (CNRS IRL2807), Universidad de Chile, Santiago, Chile.
\tt{vas.livanos@gmail.com}} \and 
Solan, Eilon%
\footnote{School of Mathematical Sciences, Tel Aviv University, Israel. \tt{eilonsolan@gmail.com}}
\and 
Solan, Ron%
\footnote{Broad Institute of Massachusetts Institute of Technology and Harvard, Cambridge, MA, USA. \tt{ron.s.solan@gmail.com}}}

\begin{document}

\maketitle 

\begin{abstract}%
We consider a basic balls-and-bins question in which you have to distribute~$n$ balls into~$k$ bins. Then, round by round, a ball is removed from a non-empty bin chosen uniformly at random. The process ends when a
single non-empty bin remains. The goal is to minimize the expected number of remaining balls.
An open problem posed by Will Ma 
asks whether the initial assignment that minimizes the expected number of remaining balls is one that is as balanced as possible.
Using a coupling argument, we answer this conjecture positively, and we discuss the case of non-uniform choice among the non-empty bins.%
\end{abstract}




\section{Introduction}

You have to distribute $n$ balls into $k$ bins. Then, round by round, a ball is removed from a non-empty bin chosen uniformly at random. The process ends when a
single non-empty bin remains. The goal is to minimize the expected number of remaining balls.
This problem was presented by Will Ma during an open problems session at the Simons Institute for the Theory of Computing.
He conjectured that 
the initial assignments that minimize the expected number of remaining balls are those that distribute the $n$ balls into the $k$ bins as evenly as possible;
in other words, the number of balls between any two bins differs by at most $1$.
The purpose of this note is to affirmatively settle this conjecture.

While this conjecture is very natural and interesting in its own right, one practical motivation comes from assortment optimization. At the beginning of the selling horizon, a seller selects the number of units to stock for each product. 
Each arriving customer makes a choice among the set of products with remaining inventory. The seller's goal
is to pick the stocking quantities to maximize the total expected revenue from the sales net of the stocking cost.
Various approximations of the optimal solution under different assumptions on the data have been provided, e.g., by
\cite{liang2021assortment,
zhang2025leveraging,
mouchtaki2026joint}.
The problem we study is a variation of this model, where products are perfect substitutes, customers are indifferent among them,
and they are happy and make a purchase whenever they have a choice of products. In this case, they randomly purchase one unit of one product. The above conjecture is equivalent to determining how to initially distribute~$n$ items across~$k$ product types to maximize the expected number of happy customers. Although unrealistic as a model of customer behavior, it is arguably 
a particularly simple and mathematically interesting model that is not yet well understood.
It is thus a relevant setting in which to develop tools for addressing more general and realistic models.

\medskip
To formalize the problem, we assume the $k$ bins are labeled by $[k]:=\{1,\dots,k\}$,
and that an initial assignment of balls to bins is given by $\vec{n}:=(n_1,\dots,n_k)\in\mathbb{N}^k$,
where $\mathbb N = \{0,1,\dots\}$. 
Instead of the process implicit in the problem description, we consider the following process:
At each 
round, a bin (possibly empty) is chosen
according to the uniform distribution on $[k]$,
independently of past choices,
and if the bin is non-empty, one ball is removed from it.
Thus, the sample space is~$[k]^\infty$
and the probability distribution $\prob$ is 
the product probability measure under which the coordinates are independent and uniformly distributed on $[k]$.
This representation is independent of the initial assignment,
as a point in the sample space determines the sequence of selected bins.

For each initial assignment $\vec{n}$,
denote by $X_{\vec{n}}$ the 
(random)
number of balls in the last non-empty bin.
Thus, if there are two bins and two balls, then $X_{(1,1)}$ and $X_{(2,0)}$ equal $1$ and $2$, respectively.

We number rounds starting at $0$: round $0$ refers to the initial configuration, and, for $t\geq 1$, the configuration at the end of round $t$ is the configuration after the first $t$ selections.
Let $T(\vec{n})$ denote the first (random)
round $t\geq 0$ at which the modified removal process started from $\vec n$ has a unique non-empty bin.
Note that $X_{\vec{n}}$ is the number of balls remaining at round $T(\vec{n})$ of the modified removal process. 
Crucially, this random variable has the same distribution as the number of balls remaining when the unmodified removal process terminates.
Define
\[
f(\vec{n}):=\mathbb{E}[X_{\vec{n}}],
\quad
\text{for all $\vec n \in \mathbb N^k$}.
\]
We abuse notation and extend, in the natural way, the domain of $f$ to $\bigcup_{m=2}^{\infty}\mathbb{N}^m$.

From the definition, we see that
$f(a,0) = a$, $f(0,b) = b$, and
\begin{equation}\label{eq:recursion}
    f(a,b) = \tfrac12 f(a-1,b)+\tfrac12 f(a,b-1), \quad \text{for all $a,b\in\mathbb{N}\setminus\{0\}$.}
\end{equation}

Our main result is:
\begin{theorem}\label{thm:main}
Fix $n,k \geq 2$. 
The
quantity
$\min\{ f(\vec{n}) \colon \vec{n}\in\mathbb{N}^k, \sum_{i\in [k]} n_i = n \}$
is attained 
only
by assignments $\vec{n}
=(n_1,\dots,n_k)
\in\mathbb{N}^k$ satisfying $\sum_{i\in [k]}n_i=n$ and $|n_i - n_j| \leq 1$ for every $i,j\in [k]$.
\end{theorem}

The result was previously proved for $k=2,3$ in~\cite{GS24}. 
For $k=3$, the proof argument is a complicated induction (on the initial number of balls $n$) of six mutually dependent claims.
Each claim establishes, under suitable conditions, an inequality
concerning the expected number of balls removed throughout the process.
The complexity of the induction suggests that any generalization would require significant effort and provide little insight.

Independently of us, \cite{zhou2026optimal} proved Theorem~\ref{thm:main} for any $k$,
by embedding the discrete-time process into a continuous-time analytical framework. 
Specifically, the embedding recasts the time at which bins empty into independent Erlang random variables.
This allows reframing the problem as one concerning order statistics of~$k$ independent Erlang variables.
This yields a closed-form expression for the marginal contribution of adding one additional ball to any given bin.
The closed form expression is then analyzed and used to establish monotonicity of the expected number of balls removed throughout the process, with respect to the initial number of balls per bin. This directly translates into the optimality of the balanced initial assignment.
As observed in 
\cite{zhou2026optimal}, one key analytical difficulty in analyzing the process is that the objective function -- the expected number of balls removed -- is not Schur-concave.
Thus, typical optimization techniques are not applicable to the problem under consideration.
Our proof uses a standard coupling argument
and is simpler and significantly shorter.

\section{Proof of the main result}
For completeness, we start by proving  
Theorem~\ref{thm:main} for the case $k=2$, which involves a coupling argument.
\begin{lemma}\label{lem:specCase}
  Fix $n\in\mathbb{N}$. 
  For every $\vec{n}=(n_1,n_2)\in\mathbb{N}^2$ such that $n_1+n_2=n$, 
  \[
  f(\lceil n/2\rceil, \lfloor n/2\rfloor) \le f(n_1,n_2),
  \]
with strict inequality whenever $|n_1 - n_2| > 1$.
\end{lemma}
\begin{proof}{Proof}
Without loss of generality, assume $n_1\ge n_2$.
If $|n_1-n_2| \leq 1$, then $n_1=\lceil n/2\rceil$, $n_2=\lfloor n/2\rfloor$, and there is nothing to prove.
In what follows, we assume $n_1\geq n_2+2$. 

It suffices to show that 
$f(n_1,n_2)
>
f(n_1-1,n_2+1)$, 
since the result then follows by repeatedly applying this inequality.
To see that the inequality holds, run the process from the initial assignment~$(n_1,n_2)$ and consider a sequence of removals in which eventually, at some round, bin 1 contains exactly one ball more than bin 2.
Denote by $t$ the first round at which this happens,
and by $n'_1$ the number of balls in bin $2$ at the \emph{end} of round $t$ (so that bin $1$ contains $n'_1+1$ at the end of that round). Starting from the initial assignment $(n_1-1,n_2+1)$, the same sequence of removals, up to round~$t$,
leaves bins 1 and 2 with $n'_1$ and $n'_1+1$ balls, respectively. 
By symmetry, in any such sequence of removals, the number of balls remaining when the process stops is identically distributed no matter if the initial assignment is~$(n_1,n_2)$ or $(n_1-1,n_2+1)$. 

Consider now the complementary situation in which a sequence of removals starting with the assignment $(n_1,n_2)$ results in bin $1$ having, throughout all rounds, at least two more balls than bin~2. 
Then, bin 1 is the last non-empty bin. Moreover, on the same sequence of removals, but starting with the assignment $(n_1-1,n_2+1)$, the process takes longer and ends with fewer balls. 
Finally, note that this complementary situation occurs with positive probability.
Indeed, this situation contains, e.g., 
the sequence of chosen bins $(2, 2, \dots, 2)$ of $n_2$ consecutive choices of bin $2$. 
\end{proof}

\begin{remark}\label{rem:specCase}
Lemma~\ref{lem:specCase} in fact yields the stronger statement that $f(n_1,n_2)>f(n_1-1,n_2+1)$ for all $\vec{n}=(n_1,n_2)\in\mathbb{N}^2$ with $n_1\ge n_2+2$.
\end{remark}

\smallskip

A simple property of $f$ is the following:
\begin{lemma}
\label{lemma:new}
$f(r,1,1) \leq f(r,1)$, for every $r \geq 1$,
with strict inequality whenever $r \geq 2$.
\end{lemma}

\begin{proof}{Proof}
We prove the lemma by induction on $r$.
For $r=1$, we have $f(1,1,1) = f(1,1) = 1$.
Assume that the lemma holds for a certain $r \geq 1$.
Then
\begin{align}
\label{equ:81}
f(r+1,1,1) &=
\frac{1}{3} f(r,1,1) + \frac{2}{3} f(r+1,1)\\
\label{equ:82}
&\leq \frac{1}{3} f(r,1) + \frac{1}{3} f(r,1) + \frac{1}{3} f(r+1,0)\\
\label{equ:83}
&< \frac{1}{2} f(r,1) + \frac{1}{2} f(r+1,0)\\
\label{equ:85}
&= f(r+1,1),
\end{align}
where Eqs.~\eqref{equ:81} and~\eqref{equ:85} hold by the recursion that $f$ satisfies,
Eq.~\eqref{equ:82} holds by the induction hypothesis and  the recursion,
Eq.~\eqref{equ:83} holds since $f(r,1) \leq r < r+1 = f(r+1,0)$.
\end{proof}

We turn to Theorem~\ref{thm:main}.
Fix an initial assignment $\vec{n}$
such that $|n_i-n_j| \geq 2$ for some $i,j$.
We assume throughout that~$n_1 \geq n_j$ for every $j\in [k]$.

Denote by $e_i = (0,\dots,0,1,0,\dots,0)$ the $i$-th standard unit vector in $\mathbb{R}^k$.
By definition, the random variable $X_{\vec{n}-e_1+e_i}$ indicates the number of balls in the last non-empty bin when the initial assignment is $\vec{n}-e_1+e_i$.
With this notation,
we will prove that, for each $i$ such that $n_1\ge n_i+2$,
\[ \E[X_{\vec{n}}] 
>
\E[X_{\vec{n}-e_1+e_i}]. \]
Note that if $n_1=n_i+1$,
then by symmetry $\E[X_{\vec{n}}] = \E[X_{\vec{n}-e_1+e_i}]$.
In what follows, fix $i$ such that~$n_1\geq n_i+2$.

Given an initial assignment, 
for every point in our sample space we can calculate the number of balls in each bin at every round.
It will be convenient to assume that if bin $j$ was selected $k_{j,t}$ times up to round $t$,
and if $k_{j,t} > n_j$,
then at the end of round $t$ the bin contains a negative number of balls, that is, $n_j - k_{j,t}$ balls. 

If $\vec n$ has only one non-empty coordinate, transferring one ball from that bin to bin $i$ strictly decreases the objective
because $f(r-1,1) \leq r-1 < r = f(r,0)$. We may therefore assume henceforth that at least two bins are initially non-empty.

Central to our proof is the following stopping time. Given an initial assignment $\vec n$, let $S(\vec n)$ be the smallest $t\geq 0$ such that, at the end of round $t$, one of the following conditions holds:\footnote{Although $S(\vec n)$
depends on $i$, we opt not to make the dependence explicit in the notation.} 

\begin{enumerate}[label=(\alph*)]
\item\label{cond1} Bin $1$ contains exactly one ball more than bin $i$.
\item\label{cond2} There are exactly two non-empty bins,
and one of them contains exactly one ball.
Denote by~$r \geq 1$ the number of balls in the other bin.
\end{enumerate}
Thus, $S(\vec n)=0$ if 
(at least) 
one of the conditions~\ref{cond1} or~\ref{cond2}
already holds in the initial configuration.%
\footnote{Both conditions are triggered simultaneously if at round $S(\vec n)$ 
bin $i$ contains 0 balls, 
some bin $k\neq 1,i$ contains more than one ball,
bin $1$ contains two balls at the beginning of the round
and is selected at that round,
and all other bins are empty.}
Note  also that~$S(\vec n)$ is defined relative to the process started from $\vec n$, even when it is subsequently used to couple this process with the one started from $\vec n-e_1+e_i$.

From the definition of $S(\vec{n})$ we have $T(\vec{n})\geq S(\vec{n})$. 
We next show that the same lower bound holds for $T(\vec{n}-e_1+e_i)$. We include this result primarily for clarity, to help discard infeasible events.
\begin{lemma} 
$T(\vec{n}-e_1+e_i)\ge S(\vec{n})$.
\end{lemma}
\begin{proof}{Proof}
If $S(\vec n)=0$, the result is immediate. Assume henceforth that $S(\vec n)\geq 1$.

Take a specific sequence of bins that were selected along the process.
If condition~\ref{cond1} is triggered,
then at round~$S(\vec{n})$, 
bin $1$ has exactly one more ball than bin $i$.
Hence, in the previous round, bin $1$ contained exactly two more balls than bin $i$, and a ball was removed from bin $1$ in round~$S(\vec n)$.
Since condition~\ref{cond2} has not occurred earlier,
there are at least two non-empty bins,
and, if there are exactly two non-empty bins, 
both contain at least two balls.
In both cases, when the initial assignment is~$\vec{n}-e_1+e_i$,
the removal process does not end at round $S(\vec{n})$,
and hence $T(\vec{n}-e_1+e_i) > S(\vec{n})$.

If condition~\ref{cond2} is triggered, 
then $T(\vec{n}-e_1+e_i)$ may be smaller than~$S(\vec{n})$ only if bin $1$ contains one ball at round~$S(\vec{n})$,
so that when the initial assignment is $\vec{n}-e_1+e_i$,
it would contain no balls (if bin $i$ contained more than 1 ball, then for~\ref{cond2} to be triggered, bin 1 must contain 1 ball and condition~\ref{cond1} would have triggered before).
Since condition~\ref{cond2} was not triggered before,
either at round $S(\vec{n})$ the last ball from a \emph{third} bin $j$ has been removed,
or at round $S(\vec{n})$ the penultimate ball from bin $1$ was removed.
In both cases, when the initial assignment is $\vec{n}-e_1+e_i$,
the removal process has not ended before round $S(\vec{n})$,
and hence $T(\vec{n}-e_1+e_i) \geq S(\vec{n})$.
\end{proof}

\medskip
We will consider the following events:
\begin{itemize}
\item 
$C_a$: At round $S(\vec{n})$ condition~\ref{cond1} occurs, that is, bin~1 
 contains exactly one ball more than bin $i$.
 \end{itemize}
Outside $C_a$, when condition~\ref{cond2} first occurs, exactly two non-empty bins remain and one has one ball. We distinguish whether these bins are $\{1,i\}$, neither of them, or $\{1,j\}$ for some $j \not\in \{1,i\}$;
in the last case we further distinguish whether the occupancies are $(1,1)$, $(r,1)$, or $(1,r)$.
Note that since condition $C_a$ does not occur,
bin 1 always contains at least two balls more than bin $i$,
and therefore it cannot be that $i$ but not 1 is among the last two non-empty bins.
We therefore define:
\begin{itemize}
\item 
$F_{1,i}$:
Bins $1$ and $i$ are the last two non-empty bins.
\item
$F_\emptyset$:
The last two non-empty bins include neither $1$ nor $i$.
\item 
$F^{a,b}_{1,j}$ for $j\not\in\{1,i\}$ and $a,b\geq 1$: Bins $1$ and $j$ are the last two non-empty bins, and 
at the first round in which condition~\ref{cond2} is triggered,
bins~$1$ and~$j$ contain $a$ and $b$ balls, respectively.
\end{itemize}
Consequently, the following collection of events partitions the sample space.
\[
\calF = \bigl\{C_a, F_{1,i} \setminus C_a, F_\emptyset \setminus C_a, (F^{1,1}_{1,j} \setminus C_a)_{j \not\in\{1,i\}}, (F^{r,1}_{1,j} \setminus C_a)_{j\not\in\{1,i\}; r\geq 2}, (F^{1,r}_{1,j} \setminus C_a)_{j\not\in\{1,i\}; r\ge 2} \bigr\}.
\]
Note that $\calF$ is defined relative to the stopping time $S(\vec{n})$,
which assumes that the initial assignment is~$\vec{n}$.
Since these events cover the sample space, by the law of iterated expectations, we have
\begin{align}
\nonumber
\E[X_{\vec{n}}] 
&= \prob(C_a)\cdot\E[X_{\vec{n}} \mid C_a] 
+ \prob(F_{1,i} \setminus C_a)\cdot\E[X_{\vec{n}} \mid F_{1,i} \setminus C_a] \\
\label{equ:n}
&+ \prob(F_\emptyset \setminus C_a)\cdot\E[X_{\vec{n}} \mid F_\emptyset \setminus C_a] 
+ 
\sum_{j\not\in\{1,i\}}
\prob(F^{1,1}_{1,j} \setminus C_a)\cdot\E[X_{\vec{n}} \mid F^{1,1}_{1,j} \setminus C_a]\\
&+ \sum_{j\not\in\{1,i\}} \sum_{r \geq 2}\bigl( \prob(F^{r,1}_{1,j} \setminus C_a)\cdot\E[X_{\vec{n}} \mid F^{r,1}_{1,j} \setminus C_a] 
+ \prob(F^{1,r}_{1,j} \setminus C_a)\cdot\E[X_{\vec{n}} \mid F^{1,r}_{1,j} \setminus C_a]\bigr).
\nonumber
\end{align}
An analogous equality holds for $\E[X_{\vec{n}-e_1+e_i}]$.
Note that though the random variable $X_{\vec{n}-e_1+e_i}$ is defined relative to the initial assignment $\vec{n}-e_1+e_i$,
the partition $\calF$ which is used in the analogous equality for $\E[X_{\vec{n}-e_1+e_i}]$ is defined relative to the initial assignment $\vec{n}$.
This coupling constitutes the core of our proof strategy.

Next, we show that the first four summands in~\eqref{equ:n} are at least as large as the corresponding summands in the analogous expression for $\E[X_{\vec{n}-e_1+e_i}]$.
\begin{lemma}\label{lem:fourTerms}
The following hold:
\begin{enumerate}[label=(\roman*)]
\item\label{itm:term1} $\E[X_{\vec{n}} \mid C_a] = \E[X_{\vec{n}-e_1+e_i} \mid C_a]$.
\item\label{itm:term2} $\E[X_{\vec{n}} \mid F_{1,i}\setminus C_a] > \E[X_{\vec{n}-e_1+e_i} \mid F_{1,i}\setminus C_a]$.  
\item\label{itm:term3} $X_{\vec{n}} = X_{\vec{n}-e_1+e_i}$ on $F_\emptyset\setminus C_a$.
\item\label{itm:term4} $X_{\vec{n}} = X_{\vec{n}-e_1+e_i} = 1$ on $F^{1,1}_{1,j} \setminus C_a$, for every $j \not\in \{1,i\}$.
\end{enumerate}
\end{lemma}
\begin{proof}{Proof}
On $C_a$, 
bin $1$ contains one more ball than bin $i$ at round $S(\vec{n})$.
Hence, when the initial assignment is~$\vec{n}-e_1+e_i$,
bin $1$ contains one ball fewer,
while bin $i$ contains one additional ball at round $S(\vec{n})$.
By symmetry, $X_{\vec{n}}$ conditioned on~$C_a$ and~$X_{\vec{n}-e_1+e_i}$ conditioned on $C_a$ are identically distributed, and Item~\ref{itm:term1} follows.

On $F_{1,i} \setminus C_a$ 
the last two non-empty bins are $1$ and $i$ and condition~\ref{cond2} triggers at round $S(\vec{n})$.
Hence,
bin $1$ contains at least two more balls than bin $i$, which in turn contains one ball.
By  Remark~\ref{rem:specCase}, we have $f(m,1)>f(m-1,2)$ for every $m> 2$, and therefore Item~\ref{itm:term2} follows.

On $F_\emptyset \setminus C_a$, 
bin $1$ is empty, and hence bin $i$ contains a negative number of balls.
Even if we transferred 
in the initial assignment
a ball from bin $1$ to bin $i$,
it would not affect
the identity of the last two bins.
This establishes Item~\ref{itm:term3}.

On $F^{1,1}_{1,j} \setminus C_a$,
we have $X_{\vec{n}} = 1$.
Immediately before round $S(\vec{n})$, for some $\ell\not\in\{1,j\}$, bins~$1$,~$j$,~$\ell$ must each have contained one ball.
Indeed, if bins $1$ and $j$ had contained two and one balls,
respectively, or one and two balls, respectively, then
condition~\ref{cond2} would have occurred earlier.
Moreover, the additional non-empty bin~$\ell$
cannot be $i$, since condition~\ref{cond1} was not triggered in an earlier round.
Hence, in this case~$X_{\vec{n}-e_1+e_i}=1$.
As a result, Item~\ref{itm:term4} holds. 
\end{proof}

\medskip
Recall that $k_{\ell,t}$ denotes the number of times bin $\ell$ is selected in the first $t$ rounds and that $f(a,b)=f(b,a)$ by symmetry.
Observe that for $r \geq 2$:
\begin{align}
\E[X_{\vec{n}} \mid F^{r,1}_{1,j} \setminus C_a] =
\E[X_{\vec{n}} \mid F^{1,r}_{1,j} \setminus C_a] & = f(r,1), \notag \\
\E[X_{\vec{n}-e_1+e_i}  \mid F^{r,1}_{1,j} \setminus C_a, k_{i,S(\vec{n})}> n_i] &= f(r-1,1),  \label{eq:fourIds} \\
\E[X_{\vec{n}-e_1+e_i} \mid F^{r,1}_{1,j} \setminus C_a, k_{i,S(\vec{n})}= n_i] &= f(r-1,1,1), \notag \\
\E[X_{\vec{n}-e_1+e_i} \mid F^{1,r}_{1,j} \setminus C_a] &= f(r,0). \notag
\end{align}
Regarding the last equality,
on $F^{1,r}_{1,j} \setminus C_a$ bin $1$ contains one ball at round $S(\vec{n})$, and since~$C_a$ does not hold,
bin $i$ contains a negative number of balls.
Moreover, since condition~\ref{cond2} is triggered,
there are two possibilities for what happened at round $S(\vec n)$:
either a ball was removed from bin~1,
or a ball was removed from some bin $\ell \not\in \{1,j,i\}$.
In both cases, when the initial assignment is~$\vec n - e_1 + e_i$,
at the beginning of round $S(\vec n)$ there are $r$ balls in bin $j$ and one ball in another bin (1 or $\ell$),
and at the end of that round that other bin is empty.

Since $\prob(F_{1,i}\setminus C_a) > 0$,
Lemma~\ref{lem:fourTerms} implies that, to prove $\E[X_{\vec{n}}] 
>
\E[X_{\vec{n}-e_1+e_i}]$, it suffices to show that for every~$j\not\in\{1,i\}$ and every $r\ge 2$,
\begin{align}\label{equ:86}
& \prob(F^{r,1}_{1,j} \setminus C_a)\cdot\E[X_{\vec{n}} \mid F^{r,1}_{1,j} \setminus C_a] 
+ \prob(F^{1,r}_{1,j} \setminus C_a)\cdot\E[X_{\vec{n}} \mid F^{1,r}_{1,j} \setminus C_a] 
\notag \\ 
& \qquad \ge 
\prob(F^{r,1}_{1,j} \setminus C_a)\cdot\E[X_{\vec{n}-e_1+e_i} \mid F^{r,1}_{1,j} \setminus C_a] 
+ \prob(F^{1,r}_{1,j} \setminus C_a)\cdot\E[X_{\vec{n}-e_1+e_i} \mid F^{1,r}_{1,j} \setminus C_a].
\end{align}
Fix then $j \not\in \{1,i\}$.
By Lemma~\ref{lemma:new}, 
for every $r \geq 2$ we have
$f(r-1,1,1)\leq f(r-1,1)$, 
and hence,
by~\eqref{eq:fourIds}, Eq.~\eqref{equ:86} is implied by
\[
\bigl(\prob(F^{r,1}_{1,j} \setminus C_a)+ \prob(F^{1,r}_{1,j}
 \setminus C_a
 )\bigr)\cdot f(r,1) \ge
\prob(F^{r,1}_{1,j} \setminus C_a)\cdot f(r-1,1)
+ \prob(F^{1,r}_{1,j} \setminus C_a)\cdot f(r,0).
\]
By~\eqref{eq:recursion} and because $f(r-1,1)<r=f(r,0)$ for every $r \geq 2$, the inequality above holds provided $\prob(F^{r,1}_{1,j}\setminus C_a)\ge\prob(F^{1,r}_{1,j}\setminus C_a)$.
Thus, all that remains to establish Theorem~\ref{thm:main} is to show this last inequality. We do so in what follows.

Note that $n_\ell - k_{\ell,t}$
is the number of balls remaining in bin $\ell$ 
at the end of round $t$ 
(possibly becoming negative).
Let~$S_b$ denote the first round at which condition~\ref{cond2} holds, that is, the
first round at which exactly two bins are non-empty and one of them contains exactly one ball. 
With the above notation,
\begin{equation}\label{eq:Cac}
C_a^c=\bigl\{n_1-k_{1,t}\ge n_i-k_{i,t}+2 \text{ for every }0\le t\le S_b\bigr\}.
\end{equation}
On $C^c_a$ we have $S(\vec n)=S_b$.

We now define the reflection map 
$\Psi\colon [k]^{\infty}\to [k]^{\infty}$.
Informally,
this map swaps the selections of bins 1 and $j$ after the last round in which they contain the same number of balls, thereby sending a path in $F_{1,j}^{1,r}$ to a path in $F_{1,j}^{r,1}$ while preserving the event $C_a^c$.
Formally,
for $\omega\in [k]^{\infty}$ with~$S_b(\omega)<\infty$, let
\[
\tau(\omega):=\max\bigl\{t \colon 0\le t\le S_b(\omega), n_1-k_{1,t}(\omega)=n_j-k_{j,t}(\omega)\bigr\}
\]
whenever this set is non-empty, and set $\Psi(\omega):=\omega$ otherwise. When $\tau$ is
defined, let $\Psi(\omega)$ be the point obtained from~$\omega$ by interchanging the
symbols $1$ and $j$ in every coordinate $\omega_t$ with $t>\tau(\omega)$, all other
coordinates (in particular, all selections of bin $i$) being left unchanged.

\begin{lemma}\label{lem:involution}
The map 
$\Psi$ is a measure-preserving involution of $([k]^{\infty},\prob)$.
\end{lemma}
\begin{proof}{Proof}
The map $\Psi$ is clearly an involution, and, in particular, a bijection.
Moreover, $\Psi$ is measure preserving.
Indeed, 
for each $s \geq 0$,
on the set $\{\tau=s\}$ the function $\Psi$ interchanges the symbols 1 and $j$ after time $s$,
and hence is measure preserving;
and on the set where $\tau$ is undefined, $\Psi$ is the identity.
\end{proof}

\begin{lemma}\label{lem:conjAlt}
For all $j\not\in\{1,i\}$ and all $r\geq 2$, we have
$\prob(F_{1,j}^{r,1}\setminus C_a) \geq \prob(F_{1,j}^{1,r}\setminus C_a)$.
\end{lemma}
\begin{proof}{Proof}
Note that $\tau$ is defined on $F_{1,j}^{1,r}$.
Indeed, 
$n_1 \geq n_j$, and 
on this event
in round $S_b$ bin 1 contains less balls than bin $j$,
so there must be a round in which the two bins contain the same number of balls.
We next
assert that $\Psi(F^{1,r}_{1,j}\setminus C_a)\subseteq C^c_a$.
Indeed, let $\omega\in F^{1,r}_{1,j}\cap C^c_a$.
Observe that:
\begin{itemize}
\item 
Since $\omega$ and $\Psi(\omega)$ coincide
up to time $\tau$
and condition~\ref{cond1} is not triggered along $\omega$ up to that time (because $\omega\in C_a^c$),
it is not triggered along $\Psi(\omega)$ up to time $\tau$ as well.
\item 
Since $\omega\in F_{1,j}^{1,r}$,
in each round between $\tau+1$ and~$S_b$, bin 1 contains more balls along~$\Psi(\omega)$ than along~$\omega$. This is because $\tau$ is the last round in which the two bins contain the same number of balls, and along $\omega$, at time~$S_b$ bin $j$ contains more balls than bin $1$.
Hence, if condition~\ref{cond1} is not triggered along $\omega$ between those rounds,
it is not triggered also along~$\Psi(\omega)$.
\end{itemize}
These two items imply that
$\Psi(\omega)\in C_a^c$, thus establishing the assertion.

Next, we assert that $\Psi(F_{1,j}^{1,r})\subseteq F_{1,j}^{r,1}$. 
Indeed, 
fix $\omega \in F_{1,j}^{1,r}$.
At time $\tau(\omega)$ bins 1 and~$j$ contain the same number of balls, say $z$.
Since $\omega \in F_{1,j}^{1,r}$, along $\omega$, between rounds $\tau(\omega)+1$ and $S_b(\omega)$, the number of rounds in which bins 1 and~$j$ have been selected is $z-1$ and $z-r$, respectively.
Since up to time $\tau(\omega)$ both $\omega$ and $\Psi(\omega)$ coincide,
and after that time the symbols 1 and $j$ are swapped,~$\Psi(\omega) \in F_{1,j}^{r,1}$.

The two established claims imply $\Psi\bigl(F^{1,r}_{1,j}\setminus C_a\bigr) \subseteq F^{r,1}_{1,j}\setminus C_a$.
To conclude, observe that
\[
\prob\bigl(F^{1,r}_{1,j}\setminus C_a\bigr)
=\prob\bigl(\Psi\bigl(F^{1,r}_{1,j}\cap C_a^c\bigr)\bigr)
\le 
\prob\bigl(F^{r,1}_{1,j}\setminus C_a\bigr),
\]
where the first equality is because $\Psi$ is measure preserving and injective (Lemma~\ref{lem:involution}) and the inequality is by monotonicity of $\prob$.
\end{proof}

\section{Non-uniform selection of bins}
In the case where the bins from which balls are removed are chosen non-uniformly, it is natural to ask whether the optimal initial assignment of balls to bins is proportional to the probability of picking each bin. The goal of this section is to show that this is not the case even for the two-bin scenario.

Consider the case where $k=2$ and $n=6$.
We claim that the initial assignment $(5,1)$ is not optimal when, at each round, bins 1 and 2 are chosen with probabilities $\nicefrac{5}{6}$ and $\nicefrac{1}{6}$, respectively. 
Note that the assignment $(5,1)$ is perfectly proportional to the selection probabilities.

Let $f_p(a,b)$ denote the expected number of balls remaining when the process stops, starting from the initial assignment~$(a,b)$, when bin 1 is selected with probability $p$ and bin $2$ with probability~$1-p$ at each round.
To achieve this section's stated goal, it suffices to show that $f_{\nicefrac{5}{6}}(4,2) < f_{\nicefrac{5}{6}}(5,1)$. That is, the perfectly proportional initial assignment is not optimal: moving one ball from the heavier bin to the lighter bin reduces the expected number of balls in the last non-empty bin.

First, note that $f_p(a,b)$ satisfies the following recursion:
\begin{equation*}
f_p(a,b) = \begin{cases}
a, & \text{if $b=0$,} \\
b, & \text{if $a=0$,} \\
pf_p(a-1,b)+(1-p)f_p(a,b-1), & \text{otherwise.}
\end{cases}
\end{equation*}
Using this relation, one can verify that the following expression holds for $b=1$:
\[
f_p(a,1) = a - \frac{p-p^a}{1-p} \quad\text{for $a>0$}.
\]
Similarly, one can verify that
for $b=2$:
\[
f_p(a,2) = a + 2 - \frac{2}{1-p} + p^a\left(a + 2 + \frac{2p}{1-p}\right) \quad \text{for $a\in\mathbb{N}$}.
\]
Simple calculations show that 
\[ f_{\nicefrac{5}{6}}(5,1) = \frac{3125}{1296} \approx 2.41, \ \ \ f_{\nicefrac{5}{6}}(4,2) = \frac{139}{81} \approx 1.72. \]
In particular, the proportional initial assignment $(5,1)$ is not optimal.

While the optimal allocation is not balanced,
we conjecture that for every number of bins,
the~$L_\infty$ distance between the optimal allocation and the balanced allocation is $O(\sqrt{n})$.

\paragraph{Acknowledgment.}
We thank Bernardo Subercaseaux for calling to our attention the problem addressed in this manuscript and for several insightful conversations. 
We also thank Will Ma and Elchanan Mossel for useful discussions and pointers to the literature.
The authors acknowledge the use of Claude (Anthropic - July 2026) in identifying an error in an earlier proof of Lemma~\ref{lem:conjAlt} and in developing the argument given here.

Part of this work was done while the third author was a postdoc and the fourth author was a visitor at the Centro de Modelamiento Matemático (CMM) at Universidad de Chile.

\paragraph{Funding.}
This work was partially supported by the Israel Science Foundation, Grant \#211/22, and by ANID Chile through the Programa de Investigación Asociativa (PIA) Project FB210005  and FONDECYT 1260036.

\bibliographystyle{alpha}

\newcommand{\etalchar}[1]{$^{#1}$}

\end{document}